\documentclass[11pt,leqno]{article}
\usepackage[margin=1in]{geometry} 
\usepackage{amssymb,amsfonts,amsmath,bbm,mathrsfs,stmaryrd,mathtools}
\usepackage{xcolor}
\usepackage{url}

\usepackage{extarrows}

\usepackage[shortlabels]{enumitem}
\usepackage{tensor}

\usepackage{xr}
\usepackage[T1]{fontenc}

\usepackage[colorlinks,
linkcolor=black!75!red,
citecolor=blue,
pdftitle={},
pdfproducer={pdfLaTeX},
pdfpagemode=None,
bookmarksopen=true,
bookmarksnumbered=true,
backref=page]{hyperref}

\usepackage{tikz}
\usetikzlibrary{arrows,calc,decorations.pathreplacing,decorations.markings,intersections,shapes.geometric,through,fit,shapes.symbols,positioning,decorations.pathmorphing}

\usepackage{braket}

\usepackage[amsmath,thmmarks,hyperref]{ntheorem}
\usepackage{cleveref}

\newcommand\nthalias[1]{\AddToHook{env/#1/begin}{\crefalias{lemma}{#1}}}

\nthalias{definition}
\nthalias{example}
\nthalias{examples}
\nthalias{remark}
\nthalias{remarks}
\nthalias{convention}
\nthalias{notation}
\nthalias{construction}
\nthalias{theoremN}
\nthalias{propositionN}
\nthalias{corollaryN}
\nthalias{lemma}
\nthalias{proposition}
\nthalias{corollary}
\nthalias{theorem}
\nthalias{conjecture}
\nthalias{question}
\nthalias{assumption}

\creflabelformat{enumi}{#2#1#3}

\crefname{section}{Section}{Sections}
\crefformat{section}{#2Section~#1#3} 
\Crefformat{section}{#2Section~#1#3} 

\crefname{subsection}{\S}{\S\S}
\AtBeginDocument{%
  \crefformat{subsection}{#2\S#1#3}%
  \Crefformat{subsection}{#2\S#1#3}%
}

\crefname{subsubsection}{\S}{\S\S}
\AtBeginDocument{%
  \crefformat{subsubsection}{#2\S#1#3}%
  \Crefformat{subsubsection}{#2\S#1#3}%
}

\theoremstyle{plain}

\newtheorem{lemma}{Lemma}[section]

\newtheorem{corollary}[lemma]{Corollary}
\newtheorem{theorem}[lemma]{Theorem}

\theoremstyle{plain}
\theoremnumbering{Alph}
\newtheorem{theoremN}{Theorem}

\theoremstyle{plain}
\theorembodyfont{\upshape}
\theoremsymbol{\ensuremath{\blacklozenge}}

\newtheorem{example}[lemma]{Example}

\newtheorem{remark}[lemma]{Remark}
\newtheorem{remarks}[lemma]{Remarks}

\crefname{definition}{definition}{definitions}
\crefformat{definition}{#2definition~#1#3} 
\Crefformat{definition}{#2Definition~#1#3} 

\crefname{ex}{example}{examples}
\crefformat{example}{#2example~#1#3} 
\Crefformat{example}{#2Example~#1#3} 

\crefname{exs}{example}{examples}
\crefformat{examples}{#2example~#1#3} 
\Crefformat{examples}{#2Example~#1#3} 

\crefname{remark}{remark}{remarks}
\crefformat{remark}{#2remark~#1#3} 
\Crefformat{remark}{#2Remark~#1#3} 

\crefname{remarks}{remark}{remarks}
\crefformat{remarks}{#2remark~#1#3} 
\Crefformat{remarks}{#2Remark~#1#3} 

\crefname{convention}{convention}{conventions}
\crefformat{convention}{#2convention~#1#3} 
\Crefformat{convention}{#2Convention~#1#3} 

\crefname{notation}{notation}{notations}
\crefformat{notation}{#2notation~#1#3} 
\Crefformat{notation}{#2Notation~#1#3} 

\crefname{table}{table}{tables}
\crefformat{table}{#2table~#1#3} 
\Crefformat{table}{#2Table~#1#3}

\crefname{lemma}{lemma}{lemmas}
\crefformat{lemma}{#2lemma~#1#3} 
\Crefformat{lemma}{#2Lemma~#1#3} 

\crefname{proposition}{proposition}{propositions}
\crefformat{proposition}{#2proposition~#1#3} 
\Crefformat{proposition}{#2Proposition~#1#3} 

\crefname{propositionN}{proposition}{propositions}
\crefformat{propositionN}{#2proposition~#1#3} 
\Crefformat{propositionN}{#2Proposition~#1#3} 

\crefname{corollary}{corollary}{corollaries}
\crefformat{corollary}{#2corollary~#1#3} 
\Crefformat{corollary}{#2Corollary~#1#3} 

\crefname{corollaryN}{corollary}{corollaries}
\crefformat{corollaryN}{#2corollary~#1#3} 
\Crefformat{corollaryN}{#2Corollary~#1#3} 

\crefname{theorem}{theorem}{theorems}
\crefformat{theorem}{#2theorem~#1#3} 
\Crefformat{theorem}{#2Theorem~#1#3} 

\crefname{theoremN}{theorem}{theorems}
\crefformat{theoremN}{#2theorem~#1#3} 
\Crefformat{theoremN}{#2Theorem~#1#3} 

\crefname{enumi}{}{}
\crefformat{enumi}{#2#1#3}
\Crefformat{enumi}{#2#1#3}

\crefname{assumption}{assumption}{Assumptions}
\crefformat{assumption}{#2assumption~#1#3} 
\Crefformat{assumption}{#2Assumption~#1#3} 

\crefname{construction}{construction}{Constructions}
\crefformat{construction}{#2construction~#1#3} 
\Crefformat{construction}{#2Construction~#1#3} 

\crefname{question}{question}{Questions}
\crefformat{question}{#2question~#1#3} 
\Crefformat{question}{#2Question~#1#3} 

\crefname{equation}{}{}
\crefformat{equation}{(#2#1#3)} 
\Crefformat{equation}{(#2#1#3)}

\numberwithin{equation}{section}
\renewcommand{\theequation}{\thesection-\arabic{equation}}

\theoremstyle{nonumberplain}
\theoremsymbol{\ensuremath{\blacksquare}}

\newtheorem{proof}{Proof}
\newcommand\pf[1]{\newtheorem{#1}{Proof of \Cref{#1}}}

\newcommand\bC{{\mathbb C}}
\newcommand\bD{{\mathbb D}}

\newcommand\bG{{\mathbb G}}
\newcommand\bH{{\mathbb H}}

\newcommand\bK{{\mathbb K}}
\newcommand\bL{{\mathbb L}}
\newcommand\bM{{\mathbb M}}

\newcommand\bQ{{\mathbb Q}}
\newcommand\bR{{\mathbb R}}
\newcommand\bS{{\mathbb S}}
\newcommand\bT{{\mathbb T}}

\newcommand\bZ{{\mathbb Z}}

\newcommand\cT{{\mathcal T}}

\newcommand\fg{{\mathfrak g}}

\newcommand\fj{{\mathfrak j}}
\newcommand\fk{{\mathfrak k}}
\newcommand\fl{{\mathfrak l}}

\newcommand\fs{{\mathfrak s}}

\DeclareMathOperator{\Ad}{Ad}

\DeclareMathOperator{\Aut}{\mathrm{Aut}}

\DeclareMathOperator{\im}{\mathrm{im}}

\DeclareMathOperator{\rk}{\mathrm{rk}}

\newcommand\numberthis{\addtocounter{equation}{1}\tag{\theequation}}

\newcommand{\qedhere}{\mbox{}\hfill\ensuremath{\blacksquare}}

\newcommand{\xrightarrowdbl}[2][]{%
  \xrightarrow[#1]{#2}\mathrel{\mkern-14mu}\rightarrow
}

\title{Generic infinite generation, fixed-point-poor representations and compact-element abundance in disconnected Lie groups}
\author{Alexandru Chirvasitu}

\begin{document}

\date{}

\newcommand{\Addresses}{{
  \bigskip
  \footnotesize

  \textsc{Department of Mathematics, University at Buffalo}
  \par\nopagebreak
  \textsc{Buffalo, NY 14260-2900, USA}  
  \par\nopagebreak
  \textit{E-mail address}: \texttt{achirvas@buffalo.edu}


}}

\maketitle

\begin{abstract}
  The semidirect product $\mathbb{G}=\mathbb{L}\rtimes \mathbb{K}$ attached to a compact-group action on a connected, simply-connected solvable Lie group has a dense set of compact elements precisely when the $s\in \mathbb{K}$ operating on $\mathbb{L}$ fixed-point-freely constitute a dense set. This (along with a number of alternative equivalent characterizations) extends the Wu's analogous result for connected Lie $\mathbb{K}$, and also provides ample supplies of examples of almost-connected Lie groups $\mathbb{G}$ which do not have dense sets of compact elements, even though their identity components $\mathbb{G}_0$ do. This corrects prior literature on the subject, claiming the property equivalent for $\mathbb{G}$ and $\mathbb{G}_0$.

  In a related discussion we characterize those connected Lie groups $\mathbb{G}$ with large sets of $d$-tuples generating dense subgroups $\Gamma\le \mathbb{G}$ for which the derived subgroup $\Gamma^{(1)}$ fails to be finitely-generated: $\mathbb{G}$ must either be non-trivial topologically perfect or have non-nilpotent maximal solvable quotient. 
\end{abstract}

\noindent \emph{Key words:
  Lie algebra;
  Lie group;
  derived series;
  elliptic element;
  finitely-generated;
  maximal pro-torus;
  nilpotent;
  solvable
}

\vspace{.5cm}

\noindent{MSC 2020: 17B30; 22E25; 20F16; 20F18; 22D05; 22C05; 22D12; 11R04


}


\section*{Introduction}

The problems discussed below all ultimately stem from attempts to probe how abundant \emph{elliptic} (sometimes \emph{compact} \cite[Definition 2.1]{2410.08083v1}) elements (i.e. \cite[Theorem IX.7.2]{helg_dglgssp} relatively-compact-subgroup generators) are in Lie or, more generally, locally compact groups. Several constructions can be placed on the question, and one version (\cite[pre Proposition 2.5]{zbMATH03654393}, \cite[Theorem 2]{zbMATH03975248}, \cite[Theorem 1.6]{MR3578406}, \cite[Theorem A]{2506.18861v2}) asks under what conditions a (connected, say) locally compact group contains a dense sub\emph{group} consisting of elliptic elements.

One way to go about proving that such a subgroup $\Gamma\le \bG$ exists in a connected Lie group $\bG$ is to argue that in fact for large $d\in \bZ_{\ge 0}$ ``most'' $d$-tuples in $\bG^d$ generate such subgroups. It becomes relevant, in this context, whether and to what extent one can rely on the first \emph{derived subgroup} $\Gamma^{(1)}$ \cite[post Corollary 2.22]{rot-gp} again being finitely-generated (see the introductory remarks to \Cref{se:fg} below).These considerations are what motivate \Cref{thn:nil.quot} below, to the effect that generically said derived finite generation fails. Some notation and terminology will help make sense of the statement. 

The notation for various characteristic series of Lie groups/algebras follows \cite[Definitions 10.1, 10.5, 10.8 and 10.9]{hm_pro-lie-bk}: single parentheses (brackets) indicate terms of the \emph{derived} (\emph{lower central}) series respectively, as in $\bG^{(n)}$ or $\fg^{[m]}$, and doubling indicates the analogously-defined \emph{closed} series (e.g. $\bG^{((n))}$ for a topological group $\bG$). The numbering always starts at 0:
\begin{equation*}
  \bG=\bG^{(0)}\ge \bG^{(1)}\cdots
  ,\quad
  \fg=\fg^{[0]}\ge \fg^{[1]}\ge \cdots
  ,\quad\text{etc.}
\end{equation*}
The index can be infinite, indicating the intersection of the successively-smaller members of the series, as in, say, $\bG^{((\infty))}:=\bigcap_n \bG^{((n))}$.

One way to formalize ``largeness'' for a subset of a topological space is to require that it be \emph{residual} \cite[\S 8.A]{kech_descr}: its complement is expressible as a countable union of nowhere-dense sets (i.e. sets whose closures have empty interiors). A subset $A\subseteq X$ of a topological space is \emph{cluster-residual} if every $x\in A$ has a neighborhood $U\ni x$ in $X$ with $A\cap U\subseteq U$ residual.

\begin{theoremN}\label{thn:nil.quot}
  The following conditions on a connected Lie group $\bG$ are equivalent.

  \begin{enumerate}[(a),wide]
  \item\label{item:thn:nil.quot:s.res} For some (equivalently, all sufficiently large) $d\in \bZ_{\ge 0}$ the set
    \begin{equation}\label{eq:fgnfg}
      \left\{\left(s_i\right)_{i=1}^d\ :\ \overline{\Gamma:=\Braket{s_i}_i}=\bG,\ \Gamma^{(1)}\text{ not f.g.}\right\}
      \subseteq
      \bG^d
    \end{equation}
    is residual in the set of topologically-generating $d$-tuples, which is itself non-empty and cluster-residual. 
    
  \item\label{item:thn:nil.quot:s.ne} For some (equivalently, all sufficiently large) $d\in \bZ_{\ge 0}$ the set \Cref{eq:fgnfg} is non-empty. 

  \item\label{item:thn:nil.quot:nnil} The largest solvable quotient $\bG/\bG^{((\infty))}$ of $\bG$ is non-nilpotent, or $\bG=\bG^{((\infty))}\ne \{1\}$. 
  \end{enumerate}
\end{theoremN}


As a variation on the theme, the material in \Cref{se:ac} aims to correct what appears to me to be a fallacious claim in prior literature on Lie groups with dense \emph{sets} (not necessarily groups) of elliptic elements. Specifically, the issue is with passing from connected to disconnected Lie groups with finitely many components.

The introductory discussion on \cite[p.869]{MR1044968} quickly reduces the problem to semidirect products $\bL\rtimes \bK$ for compact $\bK$ and connected, simply-connected, nilpotent $\bL$, but the subsequent claim (supported by \cite[Theorem 2.10]{MR1044968}) that finite-component Lie groups $\bG$ meet the condition if and only if their identity components $\bG_0$ do appears to be incorrect. \Cref{ex:semidir.z2} (in turn elaborating on \cite[Example 1.5]{2506.09642v1}) illustrates this, \Cref{re:wu.ident.err} identifies one flaw in the proof of the aforementioned in \cite[Theorem 2.10]{MR1044968}, an amended statement that will cover disconnected Lie groups is proved below in the form of \Cref{thn:k.by.l} below. In the statement
\begin{itemize}[wide]
\item superscripts $X^{\bG}$ denote fixed-point sets of actions $\bG\circlearrowright X$ and similarly for group elements $s\in \bG$, as in $X^s$;

\item this applies to adjoint self-actions of groups: $\bG^s$ is the \emph{centralizer} of $s\in \bG$ in $\bG$;

\item and `0' superscripts, as usual, denote identity connected components of topological groups. 
\end{itemize}

\begin{theoremN}\label{thn:k.by.l}
  The following conditions on a semidirect product $\bG\cong \bL\rtimes \bK$ with compact $\bK$ and $\bL$ connected, simply-connected, solvable Lie are equivalent.
  
  \begin{enumerate}[(a),wide]
  \item\label{item:thn:k.by.l:ell.dns} The set of elliptic elements is dense in $\bG$.

  \item\label{item:thn:k.by.l:k.act} The set
    \begin{equation*}
      \left\{s\in \bK\ :\ \bL^{s}=\{1\}\quad\text{and/or}\quad\fl^{s}=\{0\}\right\}
      \subseteq
      \bK
      ,\quad
      \fl:=Lie(\bL)
    \end{equation*}
    is dense.

  \item\label{item:thn:k.by.l:k0.act} For every $\sigma\in \bK$ the set
    \begin{equation*}
      \left\{
        s\in \bK
        \ :\
        \bL^{s\sigma}=\{1\}
        \ \text{and/or}\ 
        \fl^{s\sigma}=\{0\}
      \right\}
      \subseteq
      \bK_0
    \end{equation*}
    is dense (equivalently, contains 1 in its closure).

  \item\label{item:thn:k.by.l:k0.act.comm} For every $\sigma\in \bK$ the set
    \begin{equation*}
      \left\{
        s\in \left(\bK_0^{\sigma}\right)_0
        \ :\
        \bL^{s\sigma}=\{1\}
        \ \text{and/or}\ 
        \fl^{s\sigma}=\{0\}
      \right\}
      \subseteq
      \left(\bK_0^{\sigma}\right)_0
    \end{equation*}
    is dense (equivalently, contains $1$ in its closure).

  \item\label{item:thn:k.by.l:k0.act.etor} For every $\sigma\in \bK$ there is an $s$-normalized maximal pro-torus $\bT\le \bK_0$ with
    \begin{equation}\label{eq:etor}
      \left\{
        s\in \left(\bT^{\sigma}\right)_0
        \ :\
        \bL^{s\sigma}=\{1\}
        \ \text{and/or}\ 
        \fl^{s\sigma}=\{0\}
      \right\}
      \subseteq
      \left(\bT^{\sigma}\right)_0
    \end{equation}
    dense (equivalently, contains $1$ in its closure).

  \item\label{item:thn:k.by.l:k0.act.utor} For every $\sigma\in \bK$ \Cref{eq:etor} holds for all $\sigma$-invariant maximal pro-tori of $\bK_0$.
  \end{enumerate}
\end{theoremN}

Recall \cite[Definitions 9.30]{hm5} that \emph{pro-tori} are compact, connected, abelian groups; they exist and are all mutually conjugate in compact connected groups \cite[Theorem 9.32]{hm5}. It is a fact, presumably well-known and useful (albeit in weaker form) in the proof of \Cref{thn:k.by.l}, that an automorphism of a compact group always leaves invariant \emph{some} maximal pro-torus. For completeness and whatever intrinsic interest it may possess, we record this as \Cref{th:all.autos.fix.pro-tori}. 

\subsection*{Acknowledgments}

I am grateful for insightful and stimulating correspondence with E. Breuillard and T. Gelander.


\section{On and around finite generation in Lie groups}\label{se:fg}

In reference to extracting free dense subgroups from arbitrary f.g. subgroups of connected Lie groups, we remind the reader that \cite[Theorem 1.3]{zbMATH01903603} does precisely this for non-solvable connected Lie groups. The proof relies on  \cite[Proposition 2.7]{zbMATH01903603}, which seems to me to suffer from a slight (and fixable \cite{bg_err_pre}) issue.

That result considers a finitely-generated dense subgroup $\Gamma\le \bG$ of a connected Lie group and, in that generality (see that proof's second paragraph), asserts that the commutator subgroup $\Gamma^{(1)}$ is dense and \emph{again finitely-generated} in the closed derived subgroup $\bG^{((1))}\le \bG$. Density is of course not a problem, but finite generation possibly is: in general, the derived subgroup of an f.g. group need not be f.g. The preeminent examples are non-abelian free groups $\Gamma$ (which the $\Gamma$ of \cite[Proposition 2.7]{zbMATH01903603} might well be): their commutator subgroups are always free on infinitely many generators \cite[Theorem 11.48]{rot-gp}.

One can even arrange for $\Gamma$ to be solvable; simply consider the universal \emph{metabelian} \cite[Definitions pre \S 5.5.1]{rob_gps_2e_1995} group on $n\ge 2$ generators: the quotient $F_n/F_n^{(2)}$ of the free group $F_n$ on $n$ generators. See also \Cref{ex:semidir.dns} for instances of this phenomenon in the specific context of \cite[Proposition 2.7]{zbMATH01903603}, with $\Gamma$ embedded densely in a solvable connected Lie group. 

\begin{example}\label{ex:semidir.dns}
  Consider the 3-dimensional solvable Lie group $\bG:=\bR^2\rtimes \bS^1$, with the circle acting on the plane via its usual identification $\bS^1\cong \mathrm{SO}(2)$ as the rotation group. Or: $\bS^1\subset \bC$ acts on $\bC\cong \bR^2$ by multiplication. Let $\Gamma\le \bG$ be the 2-generated group
  \begin{equation}\label{eq:ex:semidir.dns:gammas}
    \Gamma:=\Braket{v,z}
    ,\quad
    0\ne v\in \bC\le \bC\rtimes \bS^1
    ,\quad
    z=\exp(2\pi i \theta)\in \bS^1\subset \bC
  \end{equation}
  for $\theta\in \bR$. The derived subgroup $\Gamma^{((1))}$ is precisely the $\bZ$-span of all $z^n v-v$, $n\in \bZ$. We have
  \begin{align*}
    \Gamma
    &=
      \left(\sum_{n\in \bZ}\bZ z^n v\right)\rtimes \left\{z^m\ :\ m\in \bZ\right\}
      \le \bC\rtimes \bS^1\\
    \Gamma^{(1)}
    &=
      \sum_{n\in \bZ} \bZ\left(v-z^n v\right),\numberthis\label{eq:ex:semidir.dns:g1}
  \end{align*}
  and hence:
  \begin{itemize}[wide]
  \item $\Gamma$ is dense in $\bG$ precisely when $\theta$ is irrational (so that $z$ is not a root of unity);
    
  \item while $\Gamma^{(1)}$ fails to be finitely-generated precisely \cite[\S 6.4, Proposition 14]{serre_rep_1977} when $z$ is not an \emph{algebraic integer}: a (strictly \cite[Remark post Lemma 1.6]{wash_cycltm_2e_1997}) stronger condition than not being a root of unity, but still a ``generic'' (i.e. residual) one.
  \end{itemize}
\end{example}

\begin{remarks}\label{res:exists.dns.drv}
  \begin{enumerate}[(1),wide]
  \item\label{item:res:exists.dns.drv:are.fg} As a follow-up on \Cref{ex:semidir.dns}, note that even when $z$ is not an algebraic integer (so that $\Gamma^{(1)}\le \bG^{((1))}$ is dense but not finitely-generated), there frequently are finitely-generated subgroups of $\Gamma^{(1)}$ dense in $\bG^{((1))}$. This is the case, for instance, whenever $z$ is transcendental. 

    Recall \cite[Exercise 49]{kap} that the \emph{rank} of an abelian group is
    \begin{equation*}
      \rk\Gamma
      :=
      \inf\left\{d\in \bZ_{\ge 0}
        \ :\
        \forall\text{ f.g. }\Lambda\le \Gamma\text{ is $d$-generated}\right\}.
    \end{equation*}
    $z$ being transcendental, subgroups $\Braket{F}\le \Gamma^{(1)}$ generated by sufficiently large finite subsets
    \begin{equation*}
      F\subseteq \left\{v-z^n v\ :\ n\in \bZ\right\}
    \end{equation*}
    have ranks larger than 2 so cannot be discrete. If \emph{all} failed to be dense, then the identity components $\left(\overline{\Braket{F}}\right)_0$ would stabilize, for containment-wise sufficiently large $F$, to a line $\ell\le \bC$. That line would however have to be invariant under multiplication by $z$, which is impossible.
    
    \Cref{ex:alg.not.int} below is concerned precisely with the distinction between plain algebraic and algebraic integral $z$.

  \item\label{item:res:exists.dns.drv:resid} Concerning the last point made in \Cref{ex:semidir.dns}, note also that the set of generator pairs \Cref{eq:ex:semidir.dns:gammas} for which the resulting $\Gamma$ \emph{does} contain an f.g. subgroup $\Lambda\le \Gamma^{(1)}$ dense in $\bG^{((1))}$ is residual. 
  \end{enumerate}  
\end{remarks}

As the portion of the proof of \cite[Proposition 2.7]{zbMATH01903603} affected by the apparent gap noted in the discussion preceding \Cref{ex:semidir.dns} substitutes the largest solvable quotient $\bG/\bG^{((\infty))}$ for $\bG$, one might attempt a patch by showing that whenever $\Gamma\le \bG$ is a finitely-generated dense subgroup of a connected solvable Lie group, $\Gamma^{(1)}$ at least \emph{contains} a dense subgroup of the closed derived subgroup $\bG^{((1))}\le \bG$. That this is not so is illustrated by the specialization of \Cref{ex:semidir.dns} to the ``intermediate'' case of $z$ being algebraic but not integrally so. 

\begin{example}\label{ex:alg.not.int}
  In \Cref{ex:semidir.dns}, take for $z$ any complex number of the form
  \begin{equation*}
    z:=\frac ac+\frac bc i
    ,\quad
    (a,b,c)\in \bZ^3_{>0}
    \text{ a \emph{Pythagorean triple}},
  \end{equation*}
  the latter phrase meaning (as it usually does \cite[\S 2.6]{lng_nt}) that $a^2+b^2=c^2$. A specific example would be $(a,b,c)=(3,4,5)$, mentioned also in \cite[Remark following Lemma 1.6]{wash_cycltm_2e_1997}.

  For non-zero $v\in \bC\cong \bR^2$ the elements
  \begin{equation*}
    \left\{z^n v\ :\ n\in F\right\}
    ,\quad
    \text{finite }F\subseteq\bZ
  \end{equation*}
  $\bZ$-span a discrete subgroup of $\bC$ (a lattice if $F$ is sufficiently large), so no f.g. subgroup of \Cref{eq:ex:semidir.dns:g1} will be dense in $\bG^{((1))}$. 
\end{example}

\begin{remark}\label{re:fg.nec}
  There are qualitatively different examples showing that a dense $\Gamma\le \bG$ (for connected metabelian $\bG$) certainly need not contain f.g. subgroups $\Lambda\le \Gamma^{(1)}$ dense in $\bG^{((1))}$ if $\Gamma$ was not dense to begin with.
  
  The \emph{Heisenberg group} \cite[\S 10.1]{de}
  \begin{equation*}
    \bH(\bR)
    :=
    \left\{
      \begin{pmatrix}
        1&a&b\\
        0&1&c\\
        0&0&1
      \end{pmatrix}
      \ :\
      a,b,c\in \bR
    \right\}
    \le
    \mathrm{GL}_3(\bR)
  \end{equation*}
  has the analogously-defined $\bH(\bQ)$ as a dense subgroup, with no finitely-generated subgroup of $\bH(\bQ)^{(1)}\cong \bQ$ dense in $\bH(\bR)^{((1))}=\bH(\bR)^{(1)}\cong \bR$.
\end{remark}

The large size of the ``pathological'' set of generating tuples (i.e. those rendering $\Gamma^{(1)}$ non-f.g.) raises the natural question of when and to what extent that phenomenon obtains in general. The following result provides a classification of the connected Lie groups for which \Cref{ex:semidir.dns} replicates.

\pf{thn:nil.quot}
\begin{thn:nil.quot}
  \Cref{item:thn:nil.quot:s.res} obviously implies \Cref{item:thn:nil.quot:s.ne}, and the latter in turn implies \Cref{item:thn:nil.quot:nnil} because for \emph{nilpotent} groups finite generation transports over from $\Gamma$ to $\Gamma^{(1)}$: this is easily seen by induction on the length of the derived series $\left(\Gamma^{(n)}\right)_{n}$ upon observing that whenever $\Gamma^{(1)}$ is central the commutator map
  \begin{equation*}
    \Gamma\times \Gamma
    \ni
    (\gamma,\gamma')
    \xmapsto{\quad}
    [\gamma,\gamma']
    \in \Gamma^{(1)}\le Z(\Gamma)
  \end{equation*}
  is bilinear (i.e. factors through a morphism $\left(\Gamma/\Gamma^{(1)}\right)^{\otimes 2}\to \Gamma^{(1)}$). It thus remains to prove \Cref{item:thn:nil.quot:nnil} $\Rightarrow$ \Cref{item:thn:nil.quot:s.res}.
  
  The cluster-residual character of the set of topologically-generating $d$-tuples for $d\gg 0$ follows from \cite[Lemmas 6.3 and 6.4]{zbMATH05117945} for arbitrary connected Lie groups. The source also proves that the ``cluster'' qualifier can be dropped in the solvable case. We will thus seek to show that for $d\gg 0$ the infinite generation of $\Gamma^{(1)}$ for $d$-generated dense $\Gamma$ is a generic property.

  \begin{enumerate}[(I),wide]
  \item\textbf{: $\bG=\bG^{((\infty))}\ne \{1\}$.} The set of topologically-generating $d$-tuples in $\bG$ is open \cite[Lemma 6.4]{zbMATH05117945}, and that of tuples generating \emph{free} (necessarily non-abelian) subgroups is residual therein \cite[Theorem]{zbMATH03351914}. The conclusion follows, given that the derived group of a non-abelian free group is not f.g.. 

  \item\textbf{: $\bG/\bG^{((\infty))}$ solvable non-nilpotent.} It will be enough to assume $\bG$ metabelian, for $\fg$ will be nilpotent whenever $\fg/\fg^{(2)}$ is (\Cref{le:splice} with $\fj:=\fg^{(1)}$ and $\fk:=\fg^{(2)}$). There is also no harm in passing to the \emph{universal cover} \cite[Theorem 7.7]{lee2013introduction} of $\bG$, thus assuming $\bG^{((1))}$ and $\bG/\bG^{((1))}$ vector groups.


    Recall the observation made in the proof of \cite[Lemma 6.4]{zbMATH05117945}, to the effect that for generic $x_j\in \bG$, $1\le j\le d\gg 0$ the elements $\Ad_{x_1} x_j$ span $\bG^{((1))}\cong \bR^{\ell}$. The assumed non-nilpotence means precisely that the adjoint action of $\bG$ on $\fg$ is not unipotent, so in particular generic elements' spectra will contain transcendental elements. It follows that
    \begin{equation*}
      \left\{(x_j)_{j=1}^d\ :\ \Ad_{x_1}|_{\bG^{((1))}}\text{ not integral}\right\}
      \subseteq
      \bG^d
    \end{equation*}
    is residual, so the generic choice of $(x_j)$ will yield non-f.g. subgroups
    \begin{equation*}
      \sum_j \sum_{m=1}^{\infty} \bZ\left(\Ad_{x_1}x_j-\Ad_{x_1}^m x_j\right)
      \le
      \Braket{(x_j),\ 1\le j\le d}^{(1)}
    \end{equation*}
    precisely as in \Cref{ex:semidir.dns}. The latter (abelian) group, then, must itself be non-f.g.  \qedhere
  \end{enumerate}
\end{thn:nil.quot}


While extensions of nilpotent Lie algebras of course need not be nilpotent (e.g. non-nilpotent solvable Lie algebras in characteristic 0), it is nevertheless the case that ``splicing together'' two nilpotent algebras with ``sufficiently large overlap'' will produce a nilpotent algebra. The following simple observation, formalizing this intuition, must be well-known; it is included here for completeness, and because it seems difficult to locate in the literature in precisely this formulation.

\begin{lemma}\label{le:splice}
  A finite-dimensional Lie algebra $\fg$ over an arbitrary field is nilpotent if and only if it has nilpotent ideals $\fk\le \fj^{(1)}\le \fj$ with nilpotent quotient $\fg/\fk$. 
\end{lemma}
\begin{proof}
  Only the backward implication is of any substance. The assumed nilpotence of $\fj$ means that the lower central series $(\fj^{[n]})_n$ eventually terminates at $\{0\}$; it thus suffices \cite[\S 3.3, Theorem]{hmph_1980} to prove the adjoint action nilpotent on each $\fj^{[n]}/\fj^{[n+1]}$, the assumption being that it is so for $n=0$:
  \begin{equation*}
    \fj/\fk
    \xrightarrowdbl[\quad\text{$\fg/\fk$ nilpotent}\quad]{\quad\text{$\fg$-module surjection}\quad}
    \fj/\fj^{(1)}
    =
    \fj/\fj^{[1]}.
  \end{equation*}
  The conclusion follows inductively, given the $\fg$-module morphisms
  \begin{equation*}
    \left(\fj/\fk\right)\otimes \left(\fj^{[n]}/\fj^{[n+1]}\right)
    \xrightarrowdbl[\quad]{\quad[-,-]\quad}
    \fj^{[n+1]}
    /
    \left(
      \left[\fk,\fj^{[n]}\right]
      +
      \fj^{[n+2]}
    \right)
    =
    \fj^{[n+1]}
    /    
    \fj^{[n+2]}:
  \end{equation*}
  every $x\in \fg$ is \emph{primitive} \cite[Definition 1.3.4]{mont} in the enveloping Hopf algebra $U\fg$, and a primitive element acting nilpotently in $V$ and $W$ is easily seen to do so in $V\otimes W$.
\end{proof}

\section{(Almost-)connectedness and its effects on compact-element density}\label{se:ac}

\emph{Almost-connected} topological groups (always assumed Hausdorff) , as usual \cite[Abstract]{hm_pro-lie-struct}, are those whose quotient $\bG/\bG_0$ by the identity connected component is compact. \cite[Theorem 2.10]{MR1044968} states that an almost-connected Lie group $\bG$ has a dense set of elliptic elements (i.e. $\bG$ is \emph{almost-elliptic} in the language of \cite[Definition 0.1(3)]{2506.09642v1}) if and only if its identity component $\bG_0$ does (a note added in proof also credits \cite{zbMATH03747435} with this result, but I cannot find this statement there). The non-obvious implication
\begin{equation*}
  \bG_0\text{ almost-elliptic}
  \xRightarrow{\quad}
  \bG\text{ almost-elliptic}
\end{equation*}
appears to me to be false; this is noted in passing in \cite[Remark 1.6 and Example 1.5]{2506.09642v1}, and \Cref{ex:semidir.z2} below instantiates that phenomenon more explicitly.

\begin{example}\label{ex:semidir.z2}
  Take for $\bG$ the two-component group $V\rtimes \left(\bS^1\rtimes \bZ/2\right)$, where
  \begin{itemize}[wide]
  \item $V\cong \bC^2$ is the 2-dimensional representation $\chi_1\oplus \chi_{-1}$ of the circle factor $\bS^1$, with
    \begin{equation*}
      \chi_n
      :=
      \left(\bS^1\ni z\xmapsto{\quad}z^n\in \bS^1\right)
      \in
      \widehat{\bS^1}\cong \bZ
    \end{equation*}
    denoting the circle's characters;

  \item and $\bZ/2=\Braket{\sigma}$ operating on $\bS^1$ by inversion and on $V$ by interchanging two fixed basis elements $v_{\pm}\in V_{\pm}$, the latter being the summands of $V$ respectively carrying the characters $\chi_{\pm}$. 
  \end{itemize}
  The maximal compact subgroups of $\bG$ are \cite[\S VII.3.2, Proposition 3]{bourb_int_7-9_en} the conjugates
  \begin{equation}\label{eq:wtw}
    \left\{\left(w-tw,t\right)\in V\rtimes \left(\bS^1\rtimes \bZ/2\right)\ :\ t\in \bS^1\rtimes \bZ/2\right\}
    \quad\text{for varying}\quad
    w\in V,
  \end{equation}
  so the closure of the set of elliptic elements consists of $(v,t)$ with $v\in \im (1-t)$. As all $t$ in the non-identity component $\bS^1 \sigma$ of $\bS^1\rtimes \bZ/2$ are mutually-conjugate involutions with (complex-)1-dimensional eigenspaces, no
  \begin{equation*}
    (v,t)\in V\times \bS^1\sigma
    ,\quad
    tv=v\ne 0
  \end{equation*}
  is arbitrarily approachable by elements of the form \Cref{eq:wtw}. 
\end{example}

\begin{remark}\label{re:wu.ident.err}
  It will be helpful, in addition to providing \Cref{ex:semidir.z2} above, to also identify the apparent flaw in the proof of \cite[Theorem 2.10]{MR1044968}. One problem that invalidates the argument as presented occurs in the second paragraph of the proof.

  That discussion starts with a compact Lie group $\bK$ and a finite cyclic $\bD=\Braket{d}$ normalizing it (much like the $\bK:=\bS^1$ and $\bD:=\bZ/2$ of \Cref{ex:semidir.z2}). Representing $\bK\cdot \bD$ on a vector space $V$ (again, as in \Cref{ex:semidir.z2}), the claim is that for generic $t$ in a maximal torus of $\bK$ the operator $d^{-1}-t$ on $V$ is invertible (equivalently: $1-dt$ is invertible). The example already shows that this need not be so, as $1-dt$, there, are all scalar multiples of (complex) rank-1 idempotents.
  
  The issue with the proof is the claim that as soon as $d^{-1}$ and $t$ agree on some $v$, all of their respective powers $(d^{-1})^{\ell}$ and $t^{\ell}$ do as well, for $\ell\in \bZ_{>0}$. There is no reason why this would be so (unless $d$ and $t$ in fact \emph{commute}, which is not the standing assumption):
  \begin{equation*}
    \begin{pmatrix}
      0&1\\
      1&0
    \end{pmatrix}
    \quad\text{and}\quad
    \begin{pmatrix}
      z&0\\
      0&z^{-1}
    \end{pmatrix}
    ,\quad
    z\in \bS^1
  \end{equation*}
  agree on $
  \begin{pmatrix}
    x\\zx
  \end{pmatrix}
  $ for every $x\in \bC^{\times}$, while their squares do not unless $z=\pm 1$. 
\end{remark}


Recall, for context, that all locally compact groups fitting into an extension
\begin{equation*}
  \{1\}\to
  \bL
  \lhook\joinrel\xrightarrow{\quad}
  \bG
  \xrightarrowdbl{\quad}
  \bK
  \to\{1\}
\end{equation*}
with compact $\bK$ and connected, simply-connected, solvable $\bL$ split as $\bG\cong \bL\rtimes \bK$ \cite[Exercise 19 for \S III.9, p.401]{bourb_lie_1-3}. For such groups \Cref{thn:k.by.l} gives an almost-ellipticity criterion (very much in the spirit of its analogue \cite[Theorem 2.7]{MR1044968} for extensions of compact \emph{connected} Lie groups by simply-connected nilpotent ones).

\pf{thn:k.by.l}
\begin{thn:k.by.l}
  That in each of \Cref{item:thn:k.by.l:k0.act}-\Cref{item:thn:k.by.l:k0.act.etor} the two respective versions are equivalent is clear from the invariance of the sets in question under left multiplication by $\bK_0$, $\left(\bK_0^{\sigma}\right)_0$ or $\left(\bT^{\sigma}\right)_0$.
  
  \begin{enumerate}[label={},wide]
  \item\textbf{\Cref{item:thn:k.by.l:ell.dns} $\Leftrightarrow$ \Cref{item:thn:k.by.l:k.act}:} \cite[Lemma 1.7]{2506.09642v1} and the concluding portion of the proof of \cite[Theorem A]{2506.09642v1} combine to prove the equivalence, for the connectedness of $\bK$ is not essential in those arguments.

  \item\textbf{\Cref{item:thn:k.by.l:k.act} $\Leftrightarrow$ \Cref{item:thn:k.by.l:k0.act}:} This is immediate when $\bK$ is Lie (so a finite union of its connected components), which we can always assume by passing to a Lie quotient by a compact normal subgroup of $\bG$.

  \item\textbf{\Cref{item:thn:k.by.l:k0.act.etor} $\Rightarrow$ \Cref{item:thn:k.by.l:k0.act.comm} $\Rightarrow$ \Cref{item:thn:k.by.l:k0.act}} are formal.

  \item\textbf{\Cref{item:thn:k.by.l:k0.act.utor} $\Rightarrow$ \Cref{item:thn:k.by.l:k0.act.etor}:} This follows from the fact that there always are $s$-invariant maximal pro-tori in $\bK_0$:
    \begin{itemize}[wide]
    \item \cite[Theorem 7.5]{stein_endo} gives a linear-algebraic version;

    \item that transports over to compact Lie groups by the correspondence \cite[Chapter VI, \S VIII, Definition 3]{chev_lie-bk_1946} between these and complex linear algebraic groups;

    \item whence the general claim by the usual \cite[Corollary 2.36]{hm5} device of passing to a limit. 
    \end{itemize}
    
  \item\textbf{\Cref{item:thn:k.by.l:k0.act} $\Rightarrow$ \Cref{item:thn:k.by.l:k0.act.utor}:} The claim is an entirely representation-theoretic one, relegated to \Cref{le:s.cent.tor} below with that statement's $V$ in place of the original $\fl$.  \qedhere
  \end{enumerate}
\end{thn:k.by.l}

In particular, and much as \cite[Lemma 2.9]{MR1044968} suggests, in assessing whether or not an almost-connected locally compact group $\bG$ has a dense set of elliptic elements on e need only consider subgroups of the form $\bG_0\cdot \bD$ for \emph{monothetic} (i.e. \cite[p.254]{zbMATH03104699} topologically cyclic) compact $\bD\le \bG$.

\begin{corollary}\label{cor:mnth.sbgp}
  Let $\bG$ be an almost-connected locally compact group. The set of elliptic elements is dense in $\bG$ if and only if this is so for every closed subgroup of the form $\bG_0\cdot\Braket{d}$ for arbitrary $d\in \bG$. 
\end{corollary}
\begin{proof}
  The discussion on \cite[p.869]{MR1044968} explains how the problem reduces to the semidirect products covered by \Cref{thn:k.by.l} (indeed, to \emph{Lie} semidirect products of that form), and in the latter's context the claim is plain. Note also that we do not need the full force of \Cref{thn:k.by.l}: the equivalence \Cref{item:thn:k.by.l:k.act} $\Leftrightarrow$ \Cref{item:thn:k.by.l:k0.act} suffices.
\end{proof}

The preceding result notwithstanding, \Cref{ex:semidir.z2} illustrates why one cannot generally reduce almost-ellipticity even further to $\bG_0$ alone: there is a qualitative difference between $\bG_0$ and semidirect products $\bG_0\rtimes \left(\bZ/d\right)$. 

The following auxiliary observation effectively proves the implication \Cref{item:thn:k.by.l:k0.act} $\Rightarrow$ \Cref{item:thn:k.by.l:k0.act.comm} of \Cref{thn:k.by.l}, and will be used in proving the formally stronger \Cref{item:thn:k.by.l:k0.act} $\Rightarrow$ \Cref{item:thn:k.by.l:k0.act.utor}. 

\begin{lemma}\label{le:s.cent}
  For a finite-dimensional representation $\bK\xrightarrow{\rho}\mathrm{GL}(V)$ of the compact group $\bK$ we have 
  \begin{equation*}
    \overline{\left\{s\in \bK\ :\ V^s=\{0\}\right\}}
    =
    \bK
    \xRightarrow{\quad}
    \forall\left(\sigma\in \bK\right)
    \left(
      \overline{\left\{s\in \left(\bK_0^{\sigma}\right)_0\ :\ V^{s\sigma}=\{0\}\right\}}
      =
      \left(\bK_0^{\sigma}\right)_0
    \right).
  \end{equation*}
\end{lemma}
\begin{proof}
  there is no harm in assuming all groups in sight Lie. The equivalence \Cref{item:thn:k.by.l:k.act} $\Leftrightarrow$ \Cref{item:thn:k.by.l:k0.act} of \Cref{thn:k.by.l} will also be taken for granted, so that the hypothesis can be phrased as either one of these conditions, as convenient.

  We are assuming $\sigma\in \bK$ arbitrarily approximable by elements $s \sigma\in \bK$ with no 1-eigenvectors in $V$ and $s\in \bK_0$. The claim is that the $s$ can in fact be chosen in the smaller subgroup $\left(\bK_0^{\sigma}\right)_0$ instead. To see this, observe that because
  \begin{equation*}
    \ker\left(1-\Ad_{\sigma}\right)
    =
    \fk^{\sigma}
    =
    Lie\left(\bK_0^{\sigma}\right)_0
  \end{equation*}
  we have 
  \begin{equation*}
    \fk:=Lie(\bK)
    =
    \fk^{\sigma}\oplus \im\left(1-\Ad_{\sigma}\right). 
  \end{equation*}
  It follows that the map
  \begin{equation}\label{eq:tw.ad}
    \bK_0\times \bK_0^{\sigma}
    \ni
    (t,u)
    \xmapsto{\quad}
    t^{-1} \cdot u\cdot \Ad_{\sigma}(t)
    \in \bK_0
  \end{equation}
  is a \emph{submersion} \cite[p.78]{lee2013introduction} at $(1,1)$. it follows \cite[Theorem 4.26]{lee2013introduction} that \Cref{eq:tw.ad} splits locally around $1\in \bK_0$ via some smooth map $(t_{\bullet},u_{\bullet})$:
  \begin{equation*}
    \forall\left(1\sim s\in \bK_0\right)
    \quad:\quad
    t_s^{-1} \cdot u_s\cdot \Ad_{\sigma}(t_s)=s, 
  \end{equation*}
  with `$\sim$' meaning ``close to''. We can now conclude: $s\sigma$ with $1\sim s$ is $t_s$-conjugate to $u_s \sigma$ with $1\sim u_s\in\left(\bK^{\sigma}_0\right)_0$. 
\end{proof}

\begin{lemma}\label{le:s.cent.tor}
  Let $\bK\xrightarrow{\rho}\mathrm{GL}(V)$ be a finite-dimensional representation of the compact group $\bK$. If
  \begin{equation*}
    \overline{\left\{s\in \bK\ :\ V^s=\{0\}\right\}}
    =
    \bK
  \end{equation*}
  then for every $s\in \bK$ and $s$-invariant maximal pro-torus $\bT\le \bK_0$ the set \Cref{eq:etor} with $V$ in place of $\fl$ is dense in $(\bT^s)_0$.
\end{lemma}
\begin{proof}
  $\bK$ will again be Lie. If $\sigma\in \bK$ leaves $\bT$ invariant it also permutes the summands of the corresponding \emph{root-space decomposition} \cite[post Proposition 6.44]{hm5} of $\fk$. It follows that $\bT$ normalizes $\left(\bK_0^{\sigma}\right)_0$, so centralizes (and hence contains) a maximal torus therein.

  There is thus no loss of generality in substituting $\left(\bK_0^{\sigma}\right)_0\cdot \overline{\Braket{\sigma}}$ for $\bK$ and hence assuming $\sigma$ central in $\bK$. Now, though, there is little left to prove: the elements $s\sigma$, $s\in \bK_0$ with no 1-eigenvectors and approximating $\sigma$ arbitrarily are (by the centrality of $\sigma$) $\bK_0$-conjugate to analogous products with $s\in \bT$ instead simply because the conjugates of $\bT$ cover $\bK_0$ \cite[Theorem 9.32]{hm5}. 
\end{proof}

\begin{remarks}\label{res:inv.pro.tor}
  \begin{enumerate}[(1),wide]
  \item\label{item:res:inv.pro.tor:fp} There are other routes to the observation, made in passing in the course of the proof of \Cref{thn:k.by.l}, that a monothetic compact automorphism group of a compact group $\bK$ (which may as well be assumed connected) always leaves invariant some maximal pro-torus $\bT\le \bK$. A number of observations coalesce to confirm this.

    \begin{enumerate}[(I),wide]
    \item The mutual conjugacy \cite[Theorem 9.32]{hm5} of the maximal pro-tori gives the identification
      \begin{equation*}
        \cT(\bK):=\left(\left\{\text{maximal pro-tori}\right\},\ \text{\emph{Vietoris topology} \cite[\S 1.2]{ct_vietoris}}\right)
        \quad
        \cong
        \quad
        \bK/N_{\bK}(\bT),
      \end{equation*}
      with $\bT\le \bK$ a fixed maximal pro-torus and $N_{\bK}(\bullet)$ denoting normalizers.

    \item For compact $\bL\le \Aut(\bK)$ one can always recover 
      \begin{equation*}
        \begin{aligned}
          \bK
          &\cong
            \varprojlim_{\substack{\bM\trianglelefteq \bK\\\bM\text{ is $\bL$-invariant}}}\left(\text{Lie quotient }\bK/\bM\right)
          \quad\text{and}\\
          \cT(\bK)
          &\cong
            \varprojlim_{\substack{\bM\trianglelefteq \bK\\\bM\text{ is $\bL$-invariant}}}\cT\left(\bK/\bM\right)
          \quad
          \text{\cite[Lemma 9.31]{hm5}}
        \end{aligned}    
      \end{equation*}
      simply apply the usual \cite[Corollary 2.36]{hm5} Lie-group approximation result to the semidirect product $\bK\rtimes \bL$. It thus suffices to assume
      \begin{equation*}
        \bK\ \text{and}\ \bL\text{ Lie}
        \quad
        \xRightarrow[\quad\bL\text{ monothetic}\quad]{\quad\text{\cite[Theorem 4.2.4]{de}}\quad}
        \quad
        \bL\cong \bT^d\times \bZ/k.
      \end{equation*}

    \item The quotient space $\cT(\bK)\cong \bK/N_{\bK}(T)$ is \emph{$\bQ$-acyclic} in the sense of \cite[Definition IV.5.11]{bred_gt_1997}. \cite[Lemma 2, pp.50-51]{MR2624762} argues this for unitary groups $\bK:=U(n)$, and the central ingredients are present generally by \cite[Proposition 1.2]{bgg_schubert_en}: 
      \begin{itemize}
      \item the rational cohomology of $\bK/\bT$ is all even;
      \item its dimension (and hence the \emph{Euler characteristic} \cite[Theorem 2.44]{hatch_at} $\chi(\bK/\bT)$) equals
        \begin{equation*}
          |W|
          ,\quad
          W
          :=
          W_{\bK}(\bT)
          :=
          N_{\bK}(\bT)/\bT
          \quad
          \left(\text{\emph{Weyl group} \cite[Definitions 9.30]{hm5} of $\bK$}\right);
        \end{equation*}
      \item and $W$ acts on $\bK/\bT$ freely with quotient $\bK/N_{\bK}(\bT)$, so the latter must have 1-dimensional rational cohomology. 
      \end{itemize}

    \item The fixed-point set $\cT(\bK)^{\bL_0}$ is then again $\bQ$-acyclic \cite[Lemma 1, p.48]{MR2624762}.

    \item And finally, the fixed-point set
      \begin{equation*}
        \cT(\bK)^{\bL}
        =
        \left(\cT(\bK)^{\bL_0}\right)^{\bL/\bL_0}
        =
        \left(\cT(\bK)^{\bL_0}\right)^{\bZ/n}      
      \end{equation*}
      has Euler characteristic 1 (so cannot be empty) by \cite[Lemma 1]{zbMATH03476353}. 
    \end{enumerate}

  \item\label{item:res:inv.pro.tor:lie} When $\bK$ is compact and Lie one need not require the monothetic group $\bL$ to be compact: arbitrary automorphisms of $\bK$ have invariant maximal tori. Indeed, as all maximal tori contain the connected center $Z_0(\bK):=Z(\bK)_0$, we can quotient it out and hence \cite[Theorem 9.24(iv)]{hm5} assume $\bK$ semisimple. In that case, though, its automorphism group is in any case compact \cite[Theorem 6.61(vi)]{hm5}.
    
    This same observation, it turns out, applies to arbitrary compact groups (though the argument itself does not without some supplementation): per \Cref{th:all.autos.fix.pro-tori}, every automorphism of a compact group leaves invariant some maximal pro-torus.
  \end{enumerate}
\end{remarks}

It turns out that \emph{arbitrary} (possibly non-elliptic) automorphisms of compact connected groups have invariant maximal pro-tori.

\begin{theorem}\label{th:all.autos.fix.pro-tori}
  An automorphism of a compact group $\bK$ leaves invariant some maximal pro-torus $\bT\le \bK_0$. 
\end{theorem}
\begin{proof}
  We will harmlessly assume $\bK$ connected throughout, as well as \emph{semisimple} (in the sense of \cite[Definition 9.5]{hm5}): this has no effect on the argument, as observed in \Cref{res:inv.pro.tor}\Cref{item:res:inv.pro.tor:lie}. An automorphism of $\bK$ will induce one on its \emph{pro-Lie algebra} 
  \begin{equation*}
    \fk:=Lie(\bK)
    \overset{\text{\cite[\S 2]{hm_pro-lie-struct}}}{\cong}
    \prod_{j\in J}\fs_j
    ,\quad
    \text{simple compact Lie algebras }\fs_j
  \end{equation*}
  and hence on the product $\widetilde{\bK}\cong \prod_j \bS_j$ of compact, simple, simply-connected Lie groups $\bS_j$ equipped with a profinite-kernel surjection
  \begin{equation*}
    \widetilde{\bK}
    \xrightarrowdbl{\quad}
    \bK
    \quad
    \left(\text{\cite[Corollary 9.25]{hm5}, where $\widetilde{\bK}$ would be $\bK^*$}\right).
  \end{equation*}
  Because images of maximal pro-tori through surjections are again such \cite[Lemma 9.31]{hm5}, it is enough to assume $\bK$ of the form $\prod_j \bS_j$. 

  An automorphism $\alpha\in \Aut(\bK)$ will then permute the factors $\bS_j$ \cite[Proposition 3.1 and Theorem 2.10]{hm_pro-lie-struct}, hence an induced permutation $\alpha\circlearrowright J$ (denoted abusively by the same symbol as the original automorphism) and an $\alpha$-invariant decomposition
  \begin{equation*}
    \bK=\bK_{\infty} \times \bK_{<\infty}
    \quad\text{where}\quad
    \begin{aligned}
      \bK_{\infty}
      &:=
        \prod_{|\alpha j|=\aleph_0}\prod_{j'\in \alpha_j}\bS_{j'}\\
      \bK_{<\infty}
      &:=
        \prod_{|\alpha j|<\infty}\prod_{j'\in \alpha_j}\bS_{j'}.
    \end{aligned}
  \end{equation*}
  The two factors can be handled separately, and the conclusion follows.
  \begin{itemize}[wide]
  \item On $\bK_{<\infty}$ the automorphism $\alpha$ is elliptic \cite[Lemma 2.6(ii)]{hm_pro-lie-struct} so the preceding discussion (e.g. \Cref{res:inv.pro.tor}\Cref{item:res:inv.pro.tor:fp}) applies.

  \item On the other hand, $\bK_{\infty}$ plainly has an $\alpha$-invariant maximal pro-torus: $\prod_j \bT_j$, where $\bT_j\le \bS_j$ is a fixed maximal pro-torus selected arbitrarily for $j$ ranging over a system of representatives for the infinite orbits of $\alpha$ and $\bT_{\alpha j}:=\alpha \bT_j$ for all infinite-orbit $j$. 
  \end{itemize}
\end{proof}



\addcontentsline{toc}{section}{References}

\def\polhk#1{\setbox0=\hbox{#1}{\ooalign{\hidewidth
  \lower1.5ex\hbox{`}\hidewidth\crcr\unhbox0}}}

\Addresses

\end{document}